\documentclass[11pt,reqno]{article}
\usepackage[utf8]{inputenc}
\usepackage{booktabs} % for much better looking tables
\usepackage{array} % for better arrays (eg matrices) in maths
\usepackage{paralist} % very flexible & customisable lists (eg. enumerate/itemize, etc.)
\usepackage{verbatim} % adds environment for commenting out blocks of text & for better verbatim
\usepackage{amsthm}
\usepackage[table,xcdraw]{xcolor}
\usepackage[titletoc,toc,title]{appendix}
\usepackage{latexsym, amsmath, amssymb, a4, epsfig, color}
\usepackage{blindtext}
\usepackage{subcaption}
\usepackage[utf8]{inputenc}
\usepackage[export]{adjustbox}
\usepackage{wrapfig}
\usepackage{chngcntr}
\usepackage{apptools} 
\usepackage{afterpage}
\usepackage{graphicx}
\usepackage{floatrow}
\usepackage{enumitem}
\usepackage{amssymb}
\usepackage{amsmath}
\usepackage[normalem]{ulem}
\usepackage{enumitem}

\usepackage{blkarray}
\usepackage{multirow}
\usepackage{mathtools}
\usepackage{graphicx}
\usepackage{arydshln}

\AtAppendix{\counterwithin{definition}{subsection}}
\AtAppendix{\counterwithin{theorem}{subsection}}

\graphicspath{}

\newtheorem{theorem}{Theorem}[section]

\newtheorem{lemma}[theorem]{Lemma}

\newcommand{\col}[3]{ \renewcommand{\arraystretch}{#1}
                \left[\!\! \begin{array}{c} #2 \\ #3 \end{array} \!\!\right] }

\newcommand{\RR}{\mathbb{R}}
\newcommand{\CC}{\mathbb{C}}
\newcommand{\NN}{\mathbb{N}}
\newcommand{\ZZ}{\mathbb{Z}}

\newcommand{\dA}{{\partial\!A}}
\newcommand{\Om}{\Omega}
\newcommand{\Omi}{\Omega_\mathrm{i}}
\newcommand{\Ome}{\Omega_\mathrm{e}}
\newcommand{\Omext}{{\overline\Om}^c}
\newcommand{\Gi}{{\Gamma_{\!\iii}}}
\newcommand{\Ge}{{\Gamma_{\!\eee}}}
\newcommand{\Ga}{{\Gamma_{\!\alpha}}}
\newcommand{\eee}{\mathrm{e}}
\newcommand{\iii}{\mathrm{i}}

\newcommand{\ds}{\displaystyle}

\newcommand{\p}{\partial}

\newcommand{\eqnref}[1]{(\ref {#1})}

\newcommand{\beq}{\begin{equation}}
\newcommand{\eeq}{\end{equation}}
\newcommand{\be}{\begin{equation*}}
\newcommand{\ee}{\end{equation*}}

\newcommand{\half}{{\textstyle \frac{1}{2}}}

\newcommand{\Kcal}{\mathcal{K}}
\newcommand{\Scal}{\mathcal{S}}

\newcommand{\la}{\langle}
\newcommand{\ra}{\rangle}

\numberwithin{equation}{section}
\numberwithin{figure}{section}

\begin{document}

\newcommand{\TheTitle}{Spectral analysis of the Neumann-Poincar\'{e} operator for thin doubly connected domains}
\newcommand{\TheAuthors}{D. Choi and M. Lim}

\title{{\TheTitle}}

\author
{Doosung Choi\footnotemark[1]
\footnotemark[2]
\and Mikyoung Lim\thanks{\footnotesize Department of Mathematical Sciences, Korea Advanced Institute of Science and Technology,  Daejeon 34141, South Korea ({mklim@kaist.ac.kr}).}
\and Stephen P. Shipman\thanks{\footnotesize Department of Mathematics,  Louisiana State University, Baton Rouge, LA, USA ({dchoi@lsu.edu}, {shipman@lsu.edu}).}
}

\date{\today}
\maketitle

\begin{abstract}
We analyze the spectrum of the Neumann-Poincar\'{e} (NP) operator for a doubly connected domain lying between two level curves defined by a conformal mapping, where the inner boundary of the domain is of general shape.
The analysis relies on an infinite-matrix representation of the NP operator involving the Grunsky coefficients of the conformal mapping and an application of the Gershgorin circle theorem.  As the thickness of the domain shrinks to zero, the spectrum of the doubly connected domain approaches the interval $[-1/2,1/2]$ in the Hausdorff distance and the density of eigenvalues approaches that of a thin circular annulus.
\end{abstract}
\noindent {\footnotesize {\bf Mathematics Subject Classification.} {
45C05; 45P05; 35J05; 35P05; 31A10.
}}

\noindent {\footnotesize {\bf Keywords.} 
{Neumann-Poincar\'{e} operator; integral equations; potential theory; doubly connected domains; Grunsky coefficients}
}

\section{Introduction}

Let $\Omega$ be a bounded domain in the real $(x,y)$-plane with an analytic boundary $\Gamma$.  The Neumann--Poincar\'e (NP) operator associated with $\Omega$ is the boundary-integral operator
$$
\Kcal_\Gamma[\phi](x)= -\frac{1}{2\pi} \int_\Gamma \frac{\left\la x-y,\nu_y\right\ra}{|x-y|^2}\phi(y)\, d\sigma(y),\quad x\in\Gamma,
$$
in which $\nu_x$ is the unit normal vector at $x$. 
Its $L^2$-adjoint, which we call the adjoint NP operator,~is
$$
\Kcal_\Gamma^*[\phi](x)= \frac{1}{2\pi} \int_\Gamma \frac{\left\la x-y,\nu_x\right\ra}{|x-y|^2}\phi(y)\, d\sigma(y),\quad x\in\Gamma.
$$
The normal vector $\nu_x$
%and the unit tangent vector $\tau_x$ in the direction of the orientation of $\Gamma$ are such that $\nu_x\times \tau_x$ is out-of-plane.
is directed into the region exterior to~$\Omega$.
The adjoint NP operator is related to the single-layer potential
\begin{align*}
\ds&\Scal_\Gamma[\phi](x)= \frac{1}{2\pi} \int_\Gamma \ln|x-y|\phi(y)\, d\sigma(y)\quad \mbox{for }x\in\RR^2,
\end{align*}
which is a harmonic function in $\RR^2\!\setminus\!\Gamma$, by the ``jump relations" on $\Gamma$:
\beq \label{eqn:Kstarjump}
\begin{aligned}
\Scal_\Gamma[\phi]\big|^{+}&=\Scal_\Gamma[\phi]\big|^{-},\\
\frac{\partial}{\partial\nu}\Scal_\Gamma[\phi]\Big|^{\pm}&=\left(\pm\frac{1}{2}I+\Kcal^*_\Gamma\right)[\phi].
\end{aligned}
\eeq
We will denote the identity operator in any space by~$I$.
Setting $H=\Scal_\Gamma[\phi]$, the jump relations can be written equivalently~as
\begin{equation}\label{Kgeom}
\begin{aligned}
  \phi \;=\; \left[ \frac{\p H}{\p\nu} \right]_\Gamma &\;:=\; \frac{\p H}{\p\nu}\Big|^+ - \frac{\p H}{\p\nu}\Big|^- \\
  \Kcal^*_\Gamma[\phi] \;=\; \left\langle \frac{\p H}{\p\nu} \right\rangle_{\!\!\Gamma} &\;:=\; \frac{1}{2} \left( \frac{\p H}{\p\nu}\Big|^+ + \frac{\p H}{\p\nu}\Big|^- \right).
\end{aligned}
\end{equation}

The NP operator naturally arises in interface problems for the Laplace equation, and its spectral properties have been intensely studied.  
%The spectral analysis of the NP operator has applications to nanophotonics and metamaterials: Plasmon resonance occurs near the eigenvalues of the NP operator and cloaking due to anomalous localized resonance occurs at the accumulation point of eigenvalues \cite{Ammari:2013:STN, Ando:2016:PRF,Bonnetier:2012:PBG,Bonnetier:2013:SPV, Grieser:2014:PEP, Mayergoyz:2005:ERN, Milton:2006:CEA}. 
The NP operator is not symmetric on $L^2(\Gamma)$ unless $\Omega$ is a disk or a ball \cite{Lim:2015:SBI}, but it can be symmetrized using Plemelj's symmetrization principle $\Kcal_\Gamma\Scal_\Gamma=\Scal_\Gamma\Kcal^*_\Gamma$~\cite{Khavinson:2007:PVP}, whereby $\Kcal^*_\Gamma$ becomes self-adjoint on $H^{-1/2}(\Gamma)$ and its spectrum $\sigma(\Kcal_{\Gamma}^*)$ is contained in $(-1/2,1/2]$ \cite{Escauriaza:2004:TPS, Fabes:1992:SRC,Kellogg:1929:FPT}.  When $\Gamma$ is of class $C^{1,\alpha}$, the NP operator is compact and, thus, $\sigma(\Kcal_{\Gamma}^*)$ is a sequence that accumulates to $0$. In two dimensions, the so-called twin-spectrum relation holds \cite{Mayergoyz:2005:ERN}, that is, the eigenvalues come in plus-minus pairs.
%We also refer to \cite{Helsing:2017:CSN} for the classification of the spectra of the NP operator on a planar domain with corners by using the resonance. 

%The spectrum of the NP operator is known explicitly for special shapes.  For a circle, the spectrum is $\{0, 1/2\}$, and it is known for a sphere, an ellipse, and an ellipsoid \cite{Neumann:1887:MAM,Ritter:1995:SEI}.    
%{\color{red}We refer to \cite{Kang review} and the references therein for further results on the spectrum of the NP operator.}
%Chung {\itshape et.\,al.} computed eigenvalues and eigenfunctions of the NP operator associated with confocal ellipses~\cite{Chung:2014:CAL}, Ammari {\itshape et.\,al.} \cite{Ammari:2013:STN, Ammari:2014:STN2} obtained spectra for concentric disks and spheres, and Ando {\itshape et.\,al.} proved infinitely many negative eigenvalues \cite{Ando:2019:SSN} for a torus.
The spectrum of the NP operator is known explicitly for special shapes.  For a circle, it is $\{0, 1/2\}$, and it is known for a sphere, an ellipse, an ellipsoid \cite{Neumann:1887:MAM,Ritter:1995:SEI}, confocal ellipses~\cite{Chung:2014:CAL}, and concentric disks and spheres~\cite{Ammari:2013:STN, Ammari:2014:STN2}.  For a torus, the NP operator has infinitely many negative eigenvalues \cite{Ando:2019:SSN}.  
When $\Gamma$ has corners or cusps, the NP operator has essential spectrum~\cite{PerfektPutinar2017,Cha:2018:PME}, and interesting cases and phenomena have been studied, such as  %Kang {\itshape et.\,al.} derived the spectral resolution on
intersecting disks \cite{Kang:2017:SRN,Helsing:2017:CSN}, wedge~\cite{Perfekt2019} and  conical~\cite{HelsingPerfekt2018} shapes, touching disks and crescent domains~\cite{Jung:2023:SAN}, and embedded eigenvalues~\cite{LiPerfektShipman2021,LiShipman2019}.   We refer to \cite{Ammari:2018:MCM, Kang:2022:SGA,Khavinson:2007:PVP} and the references therein for more results on the spectrum of the NP operator. 

For thin domains, which is the focus of this paper, there are some interesting results. Ando {\itshape et.\,al.} \cite{Ando:2022:SSN1} showed that, for any divergent positive sequence $\{R_j\}_{j=1}^\infty$,  the union of spectra of the NP operator for planar rectangles $\Omega_j$ with the aspect ratios $R_j$ is densely distributed in $[-1/2,1/2]$, that is,
$$
\overline{\cup_{j=1}^\infty  \sigma(\Kcal_{\p \Omega_j}^*)} = [-1/2,1/2].
$$
Furthermore,  Ando {\itshape et.\,al.}~\cite{Ando:2022:SSN} obtained that the set of eigenvalues of the NP operator on a sequence of the prolate spheroids is densely distributed in the interval $[0, 1/2]$ as their eccentricities approach~$1$. They also showed that eigenvalues for the oblate ellipsoids fill up $[-1/2,1/2]$ as the ellipsoids become flatter.

In this work, we consider a doubly connected domain of general shape, whose inner and outer boundaries are two level curves defined by a conformal mapping.  %With motivation from \cite{Ando:2022:SSN, Ando:2022:SSN1},
We prove that the spectrum of the NP operator for the doubly connected domain tends to $[-1/2,1/2]$ as thickness of the shape tends to zero, in the sense of the Hausdorff distance.  Moreover, the density of eigenvalues approaches that of a thin circular annulus.
The proof applies the Gershgorin circle theorem to an infinite matrix representation of the NP operator involving the Grunsky coefficients associated with the conformal mapping.

\section{The NP operator for a doubly connected domain}

Let $\Om$ be a simply connected bounded domain in the complex $z$-plane, with conformal radius~$\gamma$.  By the Riemann mapping theorem, there is a conformal mapping $z=\Psi(w)$ taking the exterior of the disk of radius $\gamma$ in the complex $w$-plane bijectively to the exterior ${\overline\Om}^c$ of~$\Om$, such that $z\sim w$ at infinity:
\begin{equation}
  z = \Psi(w)=w+a_0 + \frac{a_1}{w} + \frac{a_2}{w^2} + \frac{a_3}{w^3} + \cdots.
\end{equation}
For $r>\gamma$, let $\Gamma_{\!r}$ denote the curve in the $z$-plane equal to the image under $\Psi$ of the circle of radius $r$ in the $w$-plane,
\begin{equation}
  \Gamma_{\!r} \;=\; \{ \Psi(w)  : |w|=r \},
\end{equation}
which is the boundary of the region
\begin{equation}
  \Omega_r \;=\; \{ \Psi(w) : |w|<r \}.
\end{equation}
Choose two real numbers $r_\iii$ and $r_\eee$ such that
\begin{equation}
  \gamma < r_\iii < r_\eee,
\end{equation}
and define inner and outer curves $\Gi=\Gamma_{\!r_\iii}$ and $\Ge=\Gamma_{\!r_\eee}$ bounding the regions $\Omi=\Omega_{r_\iii}$ and $\Ome=\Omega_{r_\eee}$, that~is,
\begin{equation}
  \Gi \;=\; \{ \Psi(w)  : |w|=r_\iii \}, \quad   \Ge \;=\; \{ \Psi(w)  : |w|=r_\eee \}.
\end{equation}
Their union $\Gi\cup\Ge$ is the boundary of a distorted annulus in the $z$-plane,
\begin{equation}
  A \;=\; \{ \Psi(w) : r_\iii < |w| < r_\eee \}.
\end{equation}
 We are interested in the thin-domain limit $r_\iii\to r_\eee$, or $r\nearrow1$, where
\beq\label{ratio}
r = \frac{r_\iii}{r_\eee}.
\eeq
An example is shown in Fig.~\ref{doublyctd}.
\begin{figure}[H]
\begin{minipage}{.45\textwidth}
\centering\includegraphics[scale=0.45, trim = 0cm 4cm 5cm 3cm, clip]{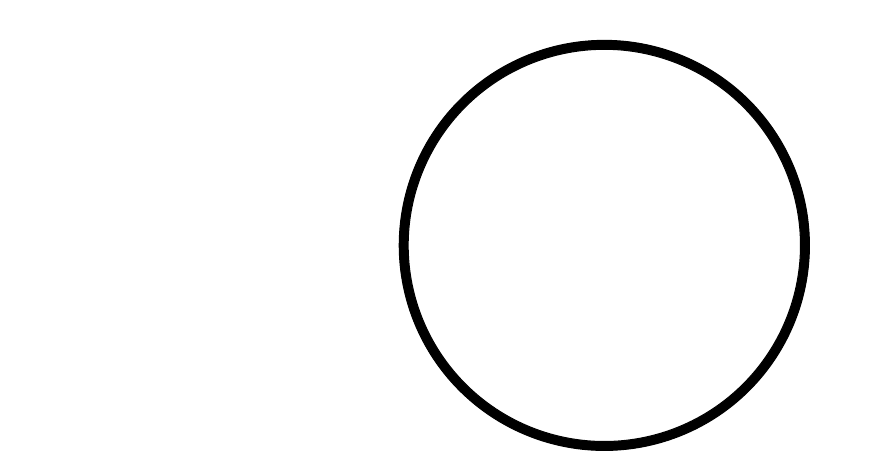}
\end{minipage}%
\begin{minipage}{.1\textwidth}
\centering$\underrightarrow{\quad\Psi\quad}$
\end{minipage}%
\begin{minipage}{.45\textwidth}
\centering\includegraphics[scale=0.45, trim = 5cm 4cm 0cm 3cm, clip]{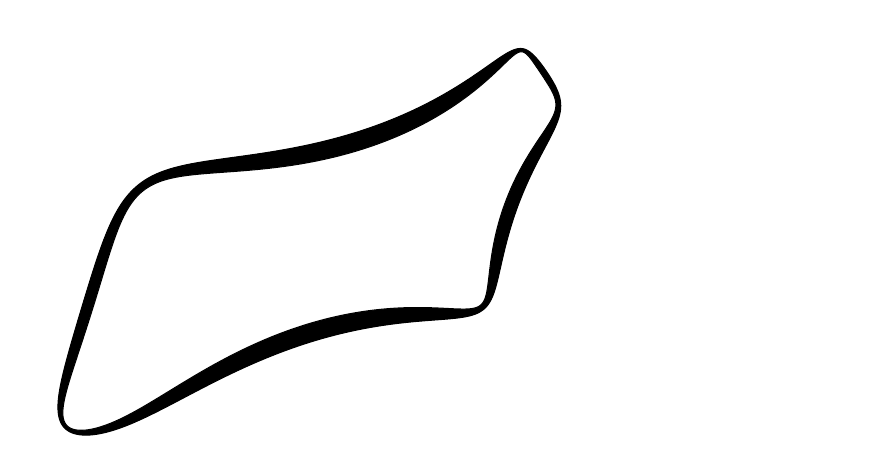}
\end{minipage}%
\caption{\small The image of a circular annulus under a conformal mapping is the thin distorted annulus.
The conformal mapping is 
$
z=\Psi(w) = w + \frac{0.3+0.5i}{w} - \frac{0.2+0.1i}{w^3} + \frac{0.1i}{w^5} + \frac{0.05i}{w^6} + \frac{0.01i}{w^7}
$
and the inner and outer radii are $r_\iii = 1.1$ and $r_\eee = 1.15$.}
\label{doublyctd}
\end{figure}

Consider the single-layer potential and the Neumann-Poincar\'e operator for the doubly connected domain~$A$.  Let the normal vector $\nu_\iii$ on $\Gi$ be taken to point into the exterior domain to~$\Gi$, and let $\nu_\eee$ on $\Ge$ point into the exterior domain to~$\Ge$.  For any density $\phi$ on $\dA$, write $\phi=\phi_\iii+\phi_\eee$, where $\phi_\iii$ is supported on $\Gi$ and $\phi_\eee$ on $\Ge$.  The single-layer potential of $\dA$ is a~sum
\begin{equation}
  \mathcal{S}_\dA[\phi] \;=\; \mathcal{S}_{\Gi}[\phi_\iii] + \mathcal{S}_{\Ge}[\phi_\eee]
\end{equation}
and consequently
\begin{align}
  \frac{\partial}{\partial\nu_\iii} \mathcal{S}_\dA[\phi]\big|^\pm_\Gi(x) &\;=\; 
           \pm\half \phi_\iii(x) + \mathcal{K}_{\Gi}^*[\phi_\iii] + \frac{\partial}{\partial\nu_\iii}\mathcal{S}_\Ge[\phi_\eee],
            \qquad x\in\Gi \\
  \frac{\partial}{\partial\nu_\eee} \mathcal{S}_\dA[\phi]\big|^\pm_\Ge(x) &\;=\;   
          \pm\half \phi_\eee(x) + \mathcal{K}_{\Ge}^*[\phi_\eee] + \frac{\partial}{\partial\nu_\eee}\mathcal{S}_\Ge[\phi_\iii],
  \qquad x\in\Ge.
\end{align}
Since the outward normal vector $\nu$ to $A$ on $\dA=\Gi\cup\Ge$ is $\nu=\nu_\eee$ on $\Ge$ and $\nu=-\nu_\iii$ on~$\Gi$, one obtains
\begin{equation*}
  \frac{\p}{\p\nu}\col{1.5}{\mathcal{S}_\dA[\phi]_\iii}{\mathcal{S}_\dA[\phi]_\eee}^\pm
  \;=\;
  \col{1.5}{-\frac{\p}{\p\nu_\iii}\mathcal{S}_\dA[\phi]\big|^\mp_\Gi}{\frac{\p}{\p\nu_\eee}\mathcal{S}_\dA[\phi]\big|^\pm_\Ge}
  \;=\;
  \col{1.5}{\pm\half\phi - \mathcal{K}^*_\Gi[\phi_\iii] - \frac{\p}{\p\nu_\iii}\mathcal{S}_\Ge[\phi_\eee]}{\pm\half\phi + \mathcal{K}^*_\Ge[\phi_\eee] + \frac{\p}{\p\nu_\eee}\mathcal{S}_\Gi[\phi_\iii]}.
\end{equation*}
In view of (\ref{eqn:Kstarjump}) therefore, with respect to the decomposition of fields on $\Ge$ and $\Gi$, the NP operator  attains the block form
\beq\label{NPdoubly}
\mathbb{K}_\dA^* =
\begin{bmatrix}
-\mathcal{K}_{\Gi}^* & -\dfrac{\p \mathcal{S}_{\Ge}}{\p \nu_\iii}\\[2mm]
\dfrac{\p \mathcal{S}_{\Gi}}{\p \nu_\eee} & \mathcal{K}_{\Ge}^*
\end{bmatrix}.
\eeq

\section{Series expansion via Grunsky coefficients}

Following \cite{Jung:2021:SEL}, this section describes how to use the Grunsky coefficients to transform the NP operator into an infinite matrix system.

Associated with the conformal mapping $z\!=\!\Psi(w)$ are the Faber polynomials $\{F_n(z)\}_{n=0}^\infty$, introduced by G. Faber \cite{Faber:1903:UPE}.  They are determined uniquely by the property that $F_n(z)-w^n = O(w^{-1})$ as $w\to\infty$, that is, there are numbers $\{c_{nm}\}_{m,n\geq1}$, known as the Grunsky coefficients, such that
\begin{equation} \label{eqn:Faberdefinition}
F_n(z)\;=\;w^n+\sum_{m=1}^{\infty}c_{nm}{w^{-m}}%\qquad \text{for} \; |w|>\gamma
\end{equation}
under the correspondence $z\!=\!\Psi(w)$, for $z\in\overline{\Omega}^c$, that is $|w|>\gamma$.  $F_n(z)$ is monic of degree~$n$, and these polynomials can be obtained recursively using the series expansion of~$\Psi$.    The first three are 
$F_0(z)=1,\ F_1(z)=z-a_0,\ F_2(z)=z^2-2a_0 z+a_0^2-2a_1.$
The Grunsky coefficients enjoy the recursion
\begin{align}\label{grunsky}
c_{n,m+1} = c_{n+1,m} - a_{n+m} + \sum_{s=1}^{n-1} a_{n-s}c_{sm} - \sum_{s=1}^{m-1} a_{m-s}c_{ns}, \quad  n,m\ge 1
\end{align}
%
%with initial values $c_{1n} = a_n$ and $c_{n1} = na_n$, $n\ge1$. 
with initial values $c_{1m} = a_m$ and $c_{m1} = ma_m$, $m\ge1$. 
Using a generating function for the Faber polynomials, one derives the relation
\begin{equation}
  mc_{nm}=nc_{mn}.
\end{equation}

The coefficients satisfy the {\it strong Grunsky inequalities} \cite{Duren:1983:UF, Grunsky:1939:KSA}: let $N\in\NN$ and $\{\lambda_k\}_{k=1}^N$ be nonzero complex numbers and it follows that
\beq\label{inequal:strong}
\sum_{n=1}^\infty n\left|\sum_{k=1}^N\frac{c_{kn}}{\gamma^{k+n}}\lambda_k \right|^2\leq\sum_{n=1}^N n\left|\lambda_n \right|^2,
\eeq
where the equality holds if and only if $\Omega$ is of measure zero. For a fixed $m$,  choose $\lambda_k=\frac{1}{\sqrt{m}}\delta_{mk}$ in \eqnref{inequal:strong} to get 
\beq\label{Gineq}
\sum_{n=1}^\infty \left|\sqrt{\frac{n}{m}}\frac{c_{mn}}{\gamma^{m+n}}\right|^2\leq 1.
\eeq

%The Faber polynomials can be equivalently defined by a generating function as follows.  For each $r>\gamma$,
%\beq\label{def:Faber}
%\frac{w\Psi'(w)}{\Psi(w)-z}=\sum_{n=0}^\infty F_n(z)\, w^{-n},\quad z\in{\overline{\Omega}_r},\;\; |w|>r.
%\eeq
%
%Making the substitution $z\mapsto\tilde{z}$ in \eqnref{def:Faber}, 
%and then putting $z=\Psi(w)$ one obtains
%%
%\beq\label{log:Faber}
%\log({z}-\tilde{z})=\log w-\sum_{n=1}^\infty \frac{1}{n}F_n(\tilde{z})\,w^{-n}
%\eeq
%%
%for $\tilde{z}\in{\overline{\Omega}_r}$ and $z\not\in{\overline{\Omega}_r}$.
%
%The expansion of (\ref{Fproperty}) in a Laurent series, 
%%
%\begin{equation} \label{eqn:Faberdefinition}
%F_n(\Psi(w))=w^n+\sum_{m=1}^{\infty}c_{nm}{w^{-m}},\quad |w|>r_\iii.
%\end{equation}
%%
%defines the well-known Grunsky coefficients $\{c_{nm}\}_{n,m=1}^\infty$, which enjoy the relation
%%
%\begin{align}\label{grunsky}
%c_{m,n+1} = c_{m+1,n} - a_{m+n} + \sum_{s=1}^{m-1} a_{m-s}c_{sn} - \sum_{s=1}^{n-1} a_{n-s}c_{ms}, \quad  m,n\ge 1
%\end{align}
%%
%%
%\begin{align}\label{grunsky}
%c_{n,m+1} = c_{n+1,m} - a_{n+m} + \sum_{s=1}^{n-1} a_{n-s}c_{sm} - \sum_{s=1}^{m-1} a_{m-s}c_{ns}, \quad  n,m\ge 1
%\end{align}
%%
%%with initial values $c_{1n} = a_n$ and $c_{n1} = na_n$, $n\ge1$. 
%with initial values $c_{1m} = a_m$ and $c_{m1} = ma_m$, $m\ge1$. 

Jung and Lim~\cite{Jung:2021:SEL} developed an explicit infinite matrix formulation of the NP operator for simply connected domains by using the Grunsky coefficients.  This formulation has been used in studies of shape recovery and neutral inclusions and in spectral analysis~\cite{Choi:2023:IPP, Choi:2021:ASR, Choi:2023:GME,Choi:2021:EEC, Ji:2022:SPN, Jung:2020:DEE}.  Below, we use it to analyze the spectrum of the NP operator for the distorted annular region~$A$.

We now give a brief description of the development in~\cite{Jung:2021:SEL}.
In the conformal image of~$\Psi$, namely $\Omext=\left\{ z=\Psi(w) : |w|>\gamma \right\}$, the equation $z\!=\!\Psi(w)$ defines $w\!=\!e^{\rho+i\theta}$ as a function of~$z$, and therefore $\rho\!>\!\log\gamma\in\RR$ and $\theta\in\RR/2\pi\ZZ$ act as global coordinates on $\Omext$.
Define $\rho_\iii$ and $\phi_\eee$ by $r_\iii=e^{\rho_\iii}$ and $r_\eee=e^{\rho_\eee}$.  
For $\alpha\!\in\!\{\iii,\eee\}$, the coordinate~$\theta$ parameterizes the curve $\Ga\!=\!\{z:\rho\equiv\rho_\alpha\}$ and the smooth functions $e^{in\theta}$ form a Riesz basis for $L^2(\Ga)$.
At the point $z\in\Ga$ corresponding to $(\rho,\theta)$, differentiation in the direction of the outward normal vector is given~by
\begin{equation}
  \frac{\p}{\p\nu} \;=\; \frac{1}{h(\rho,\theta)}\frac{\p}{\p\rho},
\end{equation}
in which $h$ is the Jacobian of~$\Psi$,
\[
h(\rho,\theta) = \left| \frac{\p}{\p \rho} \Psi(e^{\rho+i\theta}) \right| =  \left| \frac{\p}{\p \theta} \Psi(e^{\rho+i\theta})\right| = e^\rho | \Psi'(w) |.
\]

For any integer $n\geq1$, the following function $H^\alpha_n(z)$ is harmonic for $z\in\Omega_\alpha$ and for $z\in\overline{\Omega}_\alpha^c$ and decays as $z^{-1}$ at infinity.  With the identification $z=\Psi(w)$ and $w=e^{\rho+i\theta}$ in the exterior region $\overline{\Omega}^c$, define
\begin{equation}\label{Hn1}
  H^\alpha_n(z) \;:=\;
\renewcommand{\arraystretch}{1.5}
\left\{
\begin{array}{rcll}
  F_n(z), & & z\in\overline\Omega_\alpha \\
  \tilde F_n^\alpha(w) &:=\; e^{2n\rho_\alpha}\bar w^{-n} + \sum_{m=1}^\infty c_{nm}w^{-m}, &z\in\Omega_\alpha^c.
\end{array}
\right.
\end{equation}
Importantly, in the region $\overline{\Omega}^c$, which contains the boundary $\Ga$ of $\Omega_\alpha$,
\begin{equation}\label{Hn2}
  F_n(z) \;=\; w^n + \sum_{m=1}^\infty c_{nm}w^{-m}, \qquad z\in\overline\Omega_\alpha\!\setminus\!\Omega.
\end{equation}

The factor $e^{2n\rho_\alpha}$ is placed in $\tilde F_n^\alpha(w)$ so that $F_n$ and $\tilde F_n^\alpha$ coincide on $\Ga$.  Thus, $H^\alpha_n$ is the single-layer potential for the density equal to the jump of $\p H^\alpha_n/\p\nu$ across $\Ga$,
\begin{equation}\label{Hnjump}
  \left[ \frac{\p H^\alpha_n}{\p\nu} \right]_\Ga \;=\; -2n\, r_\alpha^n\frac{e^{in\theta}}{h(\rho_\alpha,\theta)}.
\end{equation}
The average of $\p H^\alpha_n/\p\nu$ across $\Ga$ is
\begin{equation}\label{Hnaverage}
  \left\langle \frac{\p H^\alpha_n}{\p\nu} \right\rangle_{\!\Ga} \;=\; -\sum_{m=1}^\infty m c_{nm} r_\alpha^{-m} \frac{e^{-im\theta}}{h(\rho_\alpha,\theta)}.
\end{equation}
For each $n\in\ZZ$, define the density function
\begin{equation*}
  \phi_n^{\alpha}(z) = \frac{e^{in\theta}}{h(\rho_\alpha,\theta)},
  \qquad z\in\Ga.
\end{equation*}

According to \cite{Jung:2021:SEL}, we can express a Hilbert space where the NP operator is acting on. Let $K^{-1/2}(\p \Gamma_\alpha)$ be a Hilbert space isometric to $l^2$ space where any complex sequence $a=(a_m)\in l^2$ corresponds to the functional $f_a\in K^{-1/2}(\p \Gamma_\alpha)$ given by
$$
f_a\left(\sum_{m\in\ZZ} b_m \phi_m^\alpha \right) = \sum_{m\in\ZZ} a_m \overline{b_m}.
$$
Thus, we represent $K^{-1/2}(\p \Gamma_\alpha)$ by
$$
K^{-1/2}(\p \Gamma_\alpha) = \left\{ \sum_{m\in\ZZ} a_m \phi_m^\alpha : \sum_{m\in\ZZ} |a_m|^2 <\infty \right\}.
$$ 
We can identify the operator $K^*_\Ga : K^{-1/2}(\p \Gamma_\alpha) \to K^{-1/2}(\p \Gamma_\alpha)$ with an infinite matrix on $\ell^2(\ZZ)$, namely,
$$
\left[K^*_\Ga \right]: \ell^2 \to \ell^2.
$$
We refer the readers to \cite{Jung:2021:SEL} for more detailed information of the spaces. We basically use this idea to find an infinite-matrix form of the integral operators in this paper.

In light of (\ref{Kgeom}), (\ref{Hnjump}), (\ref{Hnaverage}), and the relation $mc_{nm}=nc_{mn}$, one obtains, for $n\geq1$,
\begin{equation}\label{K2}
  K^*_\Ga\left[ \phi_n^\alpha \right] \;=\; \frac{1}{2} \sum_{m=1}^\infty \frac{c_{mn}}{r_\alpha^{m+n}} \phi_{-m}^\alpha,
\end{equation}
and by conjugation,
\begin{equation}\label{K3}
  K^*_\Ga\left[ \phi_{-n}^\alpha \right] \;=\; \frac{1}{2} \sum_{m=1}^\infty \frac{\overline{c_{mn}}}{r_\alpha^{m+n}} \phi_{m}^\alpha.
\end{equation}
%
%%
%\begin{equation}
%  K^*_\Ga\left[ \frac{e^{im\theta}}{h(\rho_\alpha,\theta)} \right] \;=\; \frac{1}{2} \sum_{n=1}^\infty \frac{c_{nm}}{r_\alpha^{n+m}} \frac{e^{-in\theta}}{h(\rho_\alpha,\theta)}.
%\end{equation}
%%
%By conjugation, one also has
%%
%\begin{equation}
%  K^*_\Ga\left[ \frac{e^{-im\theta}}{h(\rho_\alpha,\theta)} \right] \;=\; \frac{1}{2} \sum_{n=1}^\infty \frac{\overline{c_{nm}}}{r_\alpha^{n+m}} \frac{e^{in\theta}}{h(\rho_\alpha,\theta)}.
%\end{equation}
%%
%For $\alpha\in\{\iii,\eee\}$ and $n\in\mathbb{Z}$, define a density function $\phi_n^{\alpha}$ on $\Ga$~by
%$$
%\phi_n^{\alpha}(z) = \frac{e^{in\theta}}{h(\rho_\alpha,\theta)},
%\qquad z\in\Ga.
%$$
%According to \cite{Jung:2021:SEL}, the NP operators satisfy
The density of the single-layer potential that is constant in $\Omega_\alpha$ and equal to $\log|w|$ in $\overline{\Omega}_\alpha^c$ is $\phi_0^\alpha$, and therefore
\begin{align}\label{K1}
\mathcal{K}^*_{\Ga}[\phi_0^\alpha] = \frac{1}{2} \phi_{0}^\alpha.
\end{align}
%\begin{align}\label{K1}
%\mathcal{K}^*_{\p \Omi}[\phi_0^\iii] = \frac{1}{2} \phi_{0}^\iii, \quad \mathcal{K}^*_{\p \Ome}[\phi_{0}^\eee] = \frac{1}{2} \phi_{0}^\eee,
%\end{align}
%
%Thus, $\phi_0^\alpha$ are eigenfunctions of the extreme values of the spectrum of the NP operator, and
%\begin{align}
%&\mathcal{K}^*_{\p \Omi}[\phi_n^\iii] = \frac{1}{2} \sum_{m=1}^\infty \frac{c_{mn}}{r_\iii^{m+n}} \phi_{-m}^\iii, \quad \mathcal{K}^*_{\p \Omi}[\phi_{-n}^\iii] =\frac{1}{2} \sum_{m=1}^\infty \frac{\overline{c_{mn}}}{r_\iii^{m+n}} \phi_{m}^\iii,\label{K2}\\
%&\mathcal{K}^*_{\p \Ome}[\phi_n^\eee] = \frac{1}{2} \sum_{m=1}^\infty \frac{c_{mn}}{r_\eee^{m+n}} \phi_{-m}^\eee, \quad \mathcal{K}^*_{\p \Ome}[\phi_{-n}^\eee] = \frac{1}{2} \sum_{m=1}^\infty \frac{\overline{c_{mn}}}{r_\eee^{m+n}} \phi_{m}^\eee.\label{K3}
%\end{align}

To obtain the normal derivatives of the single-layer potentials on $\Gi$ and $\Ge$, use the second equation of (\ref{Hn1}) in the exterior region and (\ref{Hn2}) in the interior region and $\p/\p\nu=h^{-1}\p/\p\rho$.  With the observation that, for $n\geq1$,
\begin{equation}
  S_{\Ga}[\phi^\alpha_n] \;=\; -\frac{H^\alpha_n}{2nr^n_\alpha}, \qquad
  S_{\Ga}[\phi^\alpha_{-n}] \;=\; -\frac{\overline{H^\alpha_n}}{2nr^n_\alpha},  \qquad
  S_{\Ga}[\phi^\alpha_0] \;=\; \log|w| = \rho,
\end{equation}
one can compute
\begin{align}\label{S1}
\frac{\p \mathcal{S}_\Gi}{\p \nu_\eee}[\phi_0^\iii] = \phi_0^\eee, \quad \frac{\p \mathcal{S}_\Ge}{\p \nu_\iii}[\phi_0^\eee] = 0,
\end{align}
and, for $n\geq1$,
\begin{align}
\frac{\p \mathcal{S}_\Gi}{\p \nu_\eee}[\phi_n^\iii] &\;=\; \frac{1}{2} \Big(\frac{r_\iii}{r_\eee}\Big)^{\!\!n} \phi_n^\eee + \frac{1}{2} \sum_{m=1}^\infty \frac{c_{mn}}{r_\eee^m r_\iii^n} \phi_{-m}^\eee,\label{S2}\\
\frac{\p \mathcal{S}_\Gi}{\p \nu_\eee}[\phi_{-n}^\iii] &\;=\; \frac{1}{2} \Big(\frac{r_\iii}{r_\eee}\Big)^{\!\!n} \phi_{-n}^\eee + \frac{1}{2} \sum_{m=1}^\infty \frac{\overline{c_{mn}}}{r_\eee^m r_\iii^n} \phi_{m}^\eee,\label{S3}\\
\frac{\p \mathcal{S}_\Ge}{\p \nu_\iii}[\phi_n^\eee] &\;=\; - \frac{1}{2} \Big(\frac{r_\iii}{r_\eee}\Big)^{\!\!n} \phi_n^\iii + \frac{1}{2} \sum_{m=1}^\infty \frac{c_{mn}}{r_\iii^m r_\eee^n} \phi_{-m}^\iii, \label{S4}\\
\frac{\p \mathcal{S}_\Ge}{\p \nu_\iii}[\phi_{-n}^\eee] &\;=\; - \frac{1}{2} \Big(\frac{r_\iii}{r_\eee}\Big)^{\!\!n} \phi_{-n}^\iii + \frac{1}{2} \sum_{m=1}^\infty \frac{\overline{c_{mn}}}{r_\iii^m r_\eee^n} \phi_{m}^\iii.\label{S5}
\end{align}
%Let $C, G,S_1,S_2\in \CC^{\NN\times\NN}$ be semi-infinite matrices such that
Define the following semi-infinite matrices:
\begin{equation}%\label{mat:G:ep}
C
=
\begin{bmatrix}
\ c_{11} & c_{12} & c_{13}& \cdots\\[1mm]
\ c_{21} & c_{22} & c_{23} & \cdots\\[1mm]
\ c_{31} & c_{32} & c_{33} & \cdots\\
\ \vdots & \vdots & \vdots & \ddots
\end{bmatrix},
\qquad
G
=
\begin{bmatrix}
\ g_{11} & g_{12} & g_{13}& \cdots\\[1mm]
\ g_{21} & g_{22} & g_{23} & \cdots\\[1mm]
\ g_{31} & g_{32} & g_{33} & \cdots\\
\ \vdots & \vdots & \vdots & \ddots
\end{bmatrix},
%\qquad
%S =
%\left[
%\begin{array}{@{}*7r}
%\ddots& & \vdots & & \reflectbox{$\ddots$}\\
%& 0 & 0 & 1  \\[1mm]
%\hdots & 0 & 1 & 0 &\hdots\\[1mm]
%& 1 & 0 & 0 \\
%\reflectbox{$\ddots$} & & \vdots &  &\ddots
%\end{array}
%\right].
\end{equation}
where the components of $G$ are 
$$g_{mn} = \frac{c_{mn}}{r_\iii^{m+n}},$$
and $S:\NN\to-\NN$ and $S^*:-\NN\to\NN$ are
\begin{equation}\label{mat:S}
S=
\begin{bmatrix}
\ \vdots & \vdots & \vdots & \reflectbox{$\ddots$} \\[1mm]
\ 0 & 0 & 1 & \cdots\\[1mm]
\ 0 & 1 & 0 & \cdots\\[1mm]
\ 1 & 0 & 0 & \cdots
\end{bmatrix},
\qquad
S^*=
\begin{bmatrix}
\ \hdots & 0 & 0 & 1 \\[1mm]
\ \hdots & 0 & 1 & 0\\[1mm]
\ \hdots & 1 & 0 & 0\\[1mm]
\ \reflectbox{$\ddots$} & \vdots & \vdots & \vdots
\end{bmatrix}.
\end{equation}
Note that $S^*S$ is the identity on $\NN$ and $SS^*$ is the identity on $-\NN$.
Define the diagonal semi-infinite matrices
\begin{equation}%\label{mat:G:ep}
r^{\NN}
=
\left[
\begin{array}{ccccccc}
 r & 0  & 0  & \hdots \\[1mm]
0 & r^{2} & 0 & \hdots \\[1mm]
0 & 0 & r^{3} & \hdots \\
\vdots & \vdots & \vdots & \ddots
\end{array} 
\right], \quad
r^{2\NN}
=
\left[
\begin{array}{ccccccc}
 r^{2} & 0  & 0  & \hdots \\[1mm]
0 & r^{4} & 0 & \hdots \\[1mm]
0 & 0 & r^{6} & \hdots \\
\vdots & \vdots & \vdots & \ddots
\end{array} 
\right]
\end{equation}
and the infinite matrices $\mathcal{C}\in \CC^{\mathbb{Z}\times\mathbb{Z}}$,
\begin{align}
\mathcal{C}
=
\left[
\begin{array}{@{}*7r}
& & &  & \vdots & & \reflectbox{$\ddots$} \\
& 0 & 0 & 0 & c_{21} & c_{22} &\\[1mm]
& 0 & 0 & 0 & c_{11} & c_{12} & \hdots \\[1mm]
 & 0 & 0 & 1 & 0 & 0 &   \\[1mm]
\hdots & \overline{c_{12}} & \overline{c_{11}} & 0 & 0 & 0 & \\[1mm]
& \overline{c_{22}} & \overline{c_{21}} & 0 & 0 & 0 &\\
\reflectbox{$\ddots$} & & \vdots &  & & &
\end{array}
\right]
=
\begin{bmatrix}
  &  &  SC\\[1mm]
  & 1 &  \\[1mm]
 \overline{C}S^* &  &  \\
\end{bmatrix}\label{mathcalC}
\end{align}
%{\color{red}
%Then $C^- = E C^+$ with
%\beq\label{reverse_matrix}
%E = \left[
%\begin{array}{@{}*7r}
%& & & & & &\reflectbox{$\ddots$}\\
%& 0 & 0 & 0 & 0 & 1 &\\[1mm]
%& 0 & 0 & 0 & 1 & 0 &\\[1mm]
%& 0 & 0 & 1 & 0 & 0 &\\[1mm]
%& 0 & 1 & 0 & 0 & 0 &\\[1mm]
%& 1 & 0 & 0 & 0 & 0 &\\
%\reflectbox{$\ddots$}& & & & & &
%\end{array}
%\right].
%\eeq
%}
and
\begin{align}
r^{|\ZZ|}
&:=
\left[
\begin{array}{ccccccc}
\ddots & & &  & & & \reflectbox{$\ddots$} \\
& r^2 & 0  & 0  & 0  & 0  &  \\[1mm]
& 0 & r & 0  & 0  & 0  & \\[1mm]
& 0 & 0 & 1 & 0 & 0 &   \\[1mm]
& 0 & 0 & 0 & r & 0 &  \\[1mm]
& 0 & 0 & 0 & 0 & r^2 &\\
\reflectbox{$\ddots$} & & &  & & & \ddots
\end{array} 
\right]
=
\begin{bmatrix}
 S r^{\NN} S^* &  &  \\[1mm]
  & 1 &  \\[1mm]
  &  & r^{\NN} \\
\end{bmatrix}.\label{rZ}
\end{align}
Moreover, we set $\mathcal{R}_\alpha := r_\alpha^{|\ZZ|} (\alpha\in\{\iii,\eee\})$ as an infinite matrix by replacing $r$  in \eqnref{rZ} with $r_\alpha$ and then define
\beq\label{mathcalG}
\mathcal{G} := \mathcal{R}_\alpha^{-1} \mathcal{C} \mathcal{R}_\alpha^{-1}.
\eeq
Define an infinite matrix and a row vector function, $\Phi_\alpha$ ($\alpha\in\{\iii,\eee\}$), 
$$
\Phi_\alpha=
\left[
\begin{array}{ccccccc}
\hdots & \phi^\alpha_{-2}  & \phi^\alpha_{-1} & \phi^\alpha_{0} & \phi^\alpha_{1}  & \phi^\alpha_{2}  & \hdots
\end{array}
\right]
$$
Let $\Phi$ be an eigenfunction of the operator $\mathbb{K}_\dA^*$ defined by
$$
\Phi :=
\left[
\begin{array}{c}
\displaystyle\sum_{n\in\ZZ} a_n^\iii \phi_n^\iii \\[5mm]
\displaystyle\sum_{n\in\ZZ} a_n^\eee \phi_n^\eee
\end{array} 
\right].
$$
Denote by $A^\alpha$ $(\alpha\in\{\iii,\eee\})$ the infinite column vector
$$
A^\alpha = \left[
\begin{array}{ccccccc}
\hdots & a^\alpha_{-2}  & a^\alpha_{-1} & a^\alpha_{0} & a^\alpha_{1}  & a^\alpha_{2}  & \hdots
\end{array}
\right]^T.
$$
Then we obtain
\begin{align*}
\mathbb{K}_\dA^* \Phi
&=
\begin{bmatrix}
\displaystyle - \sum_{n\in\ZZ} \left( a_n^\iii \mathcal{K}_\Gi^*[\phi_n^\iii] + a_n^\eee \dfrac{\p \mathcal{S}_\Ge}{\p \nu_\iii}[\phi_n^\eee] \right)\\[2mm]
\displaystyle\sum_{n\in\ZZ} \left(a_n^\iii \dfrac{\p \mathcal{S}_\Gi}{\p \nu_\eee}[\phi_n^\iii] + a_n^\eee  \mathcal{K}_\Ge^*[\phi_n^\eee] \right)
\end{bmatrix}\\
&=\begin{bmatrix}
\displaystyle - \sum_{m\in\ZZ} \sum_{n\in\ZZ} \phi_m^\iii \left(\left[\mathcal{K}_\Gi^*\right]_{mn} a_n^\iii + \left[\dfrac{\p \mathcal{S}_\Ge}{\p \nu_\iii}\right]_{mn} a_n^\eee\right)  \\[2mm]
\displaystyle\sum_{m\in\ZZ} \sum_{n\in\ZZ} \phi_m^e \left(\left[\dfrac{\p \mathcal{S}_\Gi}{\p \nu_\eee}\right]_{mn} a_n^\iii + \left[\mathcal{K}_\Ge^*\right]_{mn} a_n^\eee\right)  
\end{bmatrix}\\
%&=\left[
%\begin{array}{c|c}
%\Phi_\iii & O \\
%\hline\rule{0pt}{2.5ex}
%O & \Phi_\eee
%\end{array}
%\right]
&=\left[
\begin{array}{c|c}
\Phi_\iii & \Phi_\eee
\end{array}
\right]
\left[
\begin{array}{c|c}
-\left[\mathcal{K}_\Gi^*\right] & -\left[\dfrac{\p \mathcal{S}_\Ge}{\p \nu_\iii}\right] \\[3mm]
\hline\rule{0pt}{4ex}
\left[\dfrac{\p \mathcal{S}_\Gi}{\p \nu_\eee}\right] & \left[\mathcal{K}_\Ge^*\right]
\end{array}
\right]
\left[\begin{array}{c}
A^\iii  \\
\hline\rule{0pt}{2.5ex}
A^\eee
\end{array}\right],
\end{align*}
%
%where $\left[\mathcal{K}_{\p \Omi}^*\right]$ and $\left[\mathcal{K}_{\p \Ome}^*\right]$ are infinite matrix forms of NP-operators.
in which the bracket notation in the blocks denotes the matrices for the operators with respect to the basis functions $\phi^{\iii,\eee}_n$.  We deduce from \eqnref{K2}--\eqnref{K1} that
\begin{align*}
&\left[\mathcal{K}_\Gi^*\right] = \frac{1}{2} \mathcal{R}_\iii^{-1} \mathcal{C} \mathcal{R}_\iii^{-1} \\
&\left[\mathcal{K}_\Ge^*\right] = \frac{1}{2} \mathcal{R}_\eee^{-1} \mathcal{C} \mathcal{R}_\eee^{-1}.
\end{align*}
Similarly, from \eqnref{S1}--\eqnref{S5} we obtain
\begin{align*}
\left[\dfrac{\p \mathcal{S}_\Gi}{\p \nu_\eee}\right] &\;=\; \frac{1}{2} \mathcal{R}_\iii \mathcal{R}_\eee^{-1} + \frac{1}{2} \mathcal{R}_\eee^{-1} \mathcal{C} \mathcal{R}_\iii^{-1}\\
\left[\dfrac{\p \mathcal{S}_\Ge}{\p \nu_\iii}\right] &\;=\; -\frac{1}{2} \mathcal{R}_\iii \mathcal{R}_\eee^{-1} + \frac{1}{2} \mathcal{R}_\iii^{-1} \mathcal{C} \mathcal{R}_\eee^{-1}.
\end{align*}
%
%From \eqnref{reverse_matrix}, we obtain
%$$
%G^- = E G^+.
%$$
Then, we have the block matrix equation
\begin{align}\label{KPhi}
\mathbb{K}_\dA^* \Phi
=\left[
\begin{array}{c|c}
\Phi_\iii & 0 \\
\hline\rule{0pt}{2.5ex}
0 & \Phi_\eee
\end{array}
\right]
\big[\mathbb{K}_\dA^*\big]
\left[\begin{array}{c}
A^\iii  \\
\hline\rule{0pt}{2.5ex}
A^\eee
\end{array}\right],
\end{align}
where $\big[\mathbb{K}_\dA^*\big]$ is a block matrix form of $\mathbb{K}_\dA^*$ given by
\begin{align}
\big[\mathbb{K}_\dA^*\big]
&:= \frac{1}{2}
\left[
\begin{array}{c|c}
- \mathcal{R}_\iii^{-1} \mathcal{C} \mathcal{R}_\iii^{-1}  & \mathcal{R}_\iii \mathcal{R}_\eee^{-1} - \mathcal{R}_\iii^{-1} \mathcal{C} \mathcal{R}_\eee^{-1}  \\
\hline\rule[-1ex]{0pt}{4ex}
\mathcal{R}_\iii \mathcal{R}_\eee^{-1} + \mathcal{R}_\eee^{-1} \mathcal{C} \mathcal{R}_\iii^{-1}  &  \mathcal{R}_\eee^{-1} \mathcal{C} \mathcal{R}_\eee^{-1}
\end{array}
\right]\notag
\\
&= \frac{1}{2}
\left[
\begin{array}{c|c}
- \mathcal{G} & r^{|\ZZ|} - \mathcal{G} r^{|\ZZ|}  \\
\hline\rule[-1ex]{0pt}{4ex}
r^{|\ZZ|} + r^{|\ZZ|} \mathcal{G}  & r^{|\ZZ|} \mathcal{G} r^{|\ZZ|}
\end{array}
\right].\label{Kmatrix}
\end{align}

\section{Spectral analysis of the Neumann-Poincar\'e operator}

%From \eqnref{KPhi}, if $\left[\begin{array}{c}
%A^\iii  \\
%\hline\rule{0pt}{2.5ex}
%A^\eee
%\end{array}\right]$ is an eigenvector of $\big[\mathbb{K}_\dA]$ with corresponding eigenvalue $\lambda$, we have

%That the eigenvalues of $\left[\mathbb{K}_\dA^*\right]$ are exactly those of $\mathbb{K}_\dA^*$ is evident from the relation
%\begin{align*}
%\mathbb{K}_\dA^* \Phi
%=
%\lambda
%\left[
%\begin{array}{c|c}
%\Phi_\iii & O \\
%\hline\rule{0pt}{2.5ex}
%O & \Phi_\eee
%\end{array}
%\right]
%\left[\begin{array}{c}
%A^\iii  \\
%\hline\rule{0pt}{2.5ex}
%A^\eee
%\end{array}\right]
%=\lambda
%\left[
%\begin{array}{c}
%\displaystyle\sum_{n\in\ZZ} a_n^\iii \phi_n^\iii \\[5mm]
%\displaystyle\sum_{n\in\ZZ} a_n^\eee \phi_n^\eee
%\end{array}
%\right]
%=\lambda \Phi.
%\end{align*}

%\begin{lemma}
%Let $\lambda\in (-1/2,1/2)$. The following statements are equivalent:
%\begin{enumerate}[label=\text{\roman*)}]
%\item For any $Y\in l_1(\ZZ) \times l_1(\ZZ)$, there exists $X\in l_1(\ZZ) \times l_1(\ZZ)$ such that $Y = \left( \lambda I - \big[\mathbb{K}_\dA^*\big] \right) X$.
%\item For any $\widetilde{Y}\in l_1(\ZZ\setminus\{0\})$, there exists $\widetilde{X}\in l_1(\ZZ\setminus\{0\})$ such that $\widetilde{Y} = \left(\lambda^2 I - \mathcal{B}\right)\widetilde{X}$.
%\end{enumerate}
%\end{lemma}

\begin{lemma}\label{injectivity}
For given $\lambda\in(-1/2,1/2)$, the operator $\left(\lambda I - \big[\mathbb{K}_\dA^*\big]\right)$ is injective if and only if $\left(\lambda^2 I - \mathcal{B}\right)$ is injective, where $\mathcal{B}$ is an operator defined by
\begin{align}\label{altmatrix}
\mathcal{B} := \frac{1}{4}
\left[
\begin{array}{c:c}
r^{2\NN} & -G(I - r^{2\NN}) \\[1mm]\hdashline\rule[-1ex]{0pt}{4ex}
-\overline{G} (I - r^{2\NN}) r^{2\NN} & r^{2\NN} + \overline{G} (I - r^{2\NN}) G (I - r^{2\NN}) \\
\end{array}
\right]
\end{align}
and any type of the identity matrix is denoted by $I$.
\end{lemma}

\begin{proof}
Let $O$ denote any zero matrices.  We first prove the injectivity of $\left(\lambda I - \big[\mathbb{K}_\dA^*\big]\right)$ implies the injectivity of $\left(\lambda^2 I - \mathcal{B}\right)$.  We assume that
\beq\label{injK}
\left( \lambda I - \big[\mathbb{K}_\dA^*\big] \right) \left[
\begin{array}{c}
X_1 \\
\hline\rule[-1ex]{0pt}{4ex}
X_2
\end{array}
\right] = \left[
\begin{array}{c}
O \\
\hline\rule[-1ex]{0pt}{4ex}
O
\end{array}
\right] 
\implies \left[
\begin{array}{c}
X_1 \\
\hline\rule[-1ex]{0pt}{4ex}
X_2
\end{array}
\right] = \left[
\begin{array}{c}
O \\
\hline\rule[-1ex]{0pt}{4ex}
O
\end{array}
\right].
\eeq
By using \eqnref{Kmatrix},  we have
\begin{align}
\left( \lambda I - \big[\mathbb{K}_\dA^*\big] \right) \left[
\begin{array}{c}
X_1 \\
\hline\rule[-1ex]{0pt}{4ex}
X_2
\end{array}
\right]
&= \left[
\begin{array}{c|c}
\lambda I + \frac{1}{2}\mathcal{G} &  - \frac{1}{2}r^{|\ZZ|} + \frac{1}{2}\mathcal{G} r^{|\ZZ|}  \\
\hline\rule[-1ex]{0pt}{4ex}
-\frac{1}{2} r^{|\ZZ|} - \frac{1}{2}r^{|\ZZ|} \mathcal{G} & \lambda I - \frac{1}{2}r^{|\ZZ|} \mathcal{G} r^{|\ZZ|}
\end{array}
\right]
\left[
\begin{array}{c}
X_1 \\
\hline\rule[-1ex]{0pt}{4ex}
X_2
\end{array}
\right] \notag\\
&= \left[
\begin{array}{c}
(\lambda I + \frac{1}{2}\mathcal{G}) X_1 - \frac{1}{2}( I - \mathcal{G}) r^{|\ZZ|} X_2  \\
\hline\rule[-1ex]{0pt}{4ex}
-\frac{1}{2} r^{|\ZZ|} (I + \mathcal{G}) X_1 + (\lambda I - \frac{1}{2}r^{|\ZZ|} \mathcal{G} r^{|\ZZ|}) X_2
\end{array}
\right].\label{P}
\end{align}
Multiplying $r^{|\ZZ|}$ on the first row of \eqnref{P} and adding it to the second row gives
\begin{align}
\left[
\begin{array}{c}
(\lambda I + \frac{1}{2}\mathcal{G}) X_1 - \frac{1}{2}( I - \mathcal{G}) r^{|\ZZ|} X_2 \\
\hline\rule[-1ex]{0pt}{4ex}
(\lambda-\frac{1}{2}) r^{|\ZZ|} X_1 + ( \lambda I - \frac{1}{2} r^{2|\ZZ|} ) X_2
\end{array}
\right]
=
\left[
\begin{array}{c}
O \\
\hline\rule[-1ex]{0pt}{4ex}
O
\end{array}
\right].\label{matrixeqn1}
\end{align}
Solving \eqnref{matrixeqn1}, we can eliminate $X_1 = -(\lambda-\frac{1}{2})^{-1} \, r^{-|\ZZ|} ( \lambda I - \frac{1}{2} r^{2|\ZZ|} ) X_2$ and get
\begin{align}
O
&=\Big( \lambda I + \frac{1}{2}\mathcal{G} \Big) \Big( \lambda I - \frac{1}{2}r^{2|\ZZ|} \Big) X_2' +  \frac{1}{2}\Big(\lambda - \frac{1}{2}\Big)( I - \mathcal{G}) r^{2|\ZZ|} X_2' \notag\\
&= \left[ \lambda^2 I + \frac{\lambda}{2} \mathcal{G} \left(I - r^{2|\ZZ|} \right)  - \frac{r^{2|\ZZ|}}{4} \right] X_2',\label{X2Y2}
\end{align}
where we denote $X_2'$ by
\beq\label{X2'}
X_2' = r^{-|\ZZ|} X_2 =:
\begin{bmatrix}
\bf{x}_- \\[1mm]
x_0 \\[1mm]
\bf{x}_+
\end{bmatrix},
\eeq
and using \eqnref{mathcalC}, \eqnref{mathcalG},  \eqnref{rZ}, we obtain an equivalent relation of \eqnref{X2Y2} as
\begin{align}
&\begin{bmatrix}
\lambda^2 I -\frac{1}{4} S r^{2\NN} S^* & O & \frac{\lambda}{2}  SG(I - r^{2\NN}) \\[1mm]
O & \lambda^2-\frac{1}{4} & O \\[1mm]
\frac{\lambda}{2} \overline{G} S^* (I - S r^{2\NN} S^*)& O & \lambda^2 I -\frac{1}{4} r^{2\NN}
\end{bmatrix}
\begin{bmatrix}
\bf{x}_- \\[1mm]
x_0 \\[1mm]
\bf{x}_+
\end{bmatrix} =  O\notag \\
\iff
&\quad 
\Big(\lambda^2 - \frac{1}{4} \Big)x_0 = 0, \quad 
\begin{bmatrix}
\lambda^2 I -\frac{1}{4} S r^{2\NN} S^*& \frac{\lambda}{2}  S G(I - r^{2\NN}) \\[1mm]
\frac{\lambda}{2} \overline{G} S^* (I - S r^{2\NN} S^*) & \lambda^2 I -\frac{1}{4} r^{2\NN}
\end{bmatrix}
\begin{bmatrix}
\bf{x}_- \\[1mm]
\bf{x}_+ 
\end{bmatrix} = O.\label{yx}
\end{align}
Since $S^*S$ and $SS^*$ are identity matrices,  the second relation of \eqnref{yx} is equivalent to
\begin{align}
&
\begin{bmatrix}
\lambda^2 I -\frac{1}{4} r^{2\NN}& \frac{1}{4} G(I - r^{2\NN}) \\[1mm]
\lambda^2 \overline{G} (I - r^{2\NN}) & \lambda^2 I -\frac{1}{4} r^{2\NN} 
\end{bmatrix}
\begin{bmatrix}
S^* \bf{x}_- \\[1mm]
2\lambda \bf{x}_+ 
\end{bmatrix} = O \notag \\
\iff
&
\left(\lambda^2 \begin{bmatrix}
I & O \\[1mm]
\overline{G} (I - r^{2\NN}) & I 
\end{bmatrix}
-
\frac{1}{4}
\begin{bmatrix}
r^{2\NN}& -G(I - r^{2\NN}) \\[1mm]
O & r^{2\NN} 
\end{bmatrix}\right)
\begin{bmatrix}
S^* \bf{x}_- \\[1mm]
2\lambda \bf{x}_+ 
\end{bmatrix} = O\notag \\
\iff
&
\left(\lambda^2 I
-
\frac{1}{4}
\begin{bmatrix}
I & O \\[1mm]
\overline{G} (I - r^{2\NN}) & I 
\end{bmatrix}^{-1}
\begin{bmatrix}
r^{2\NN}& -G(I - r^{2\NN}) \\[1mm]
O & r^{2\NN} 
\end{bmatrix}\right)
\begin{bmatrix}
S^* \bf{x}_- \\[1mm]
2\lambda \bf{x}_+ 
\end{bmatrix} = O\notag \\
\iff
&
\left(\lambda^2 I
-
\frac{1}{4}
\begin{bmatrix}
I & O \\[1mm]
-\overline{G} (I - r^{2\NN}) & I 
\end{bmatrix}
\begin{bmatrix}
r^{2\NN}& -G(I - r^{2\NN}) \notag\\[1mm]
O & r^{2\NN} 
\end{bmatrix}\right)
\begin{bmatrix}
S^* \bf{x}_- \\[1mm]
2\lambda \bf{x}_+ 
\end{bmatrix} = O\\[1mm]
\iff
& \left(\lambda^2 I - \mathcal{B} \right) \begin{bmatrix}
S^* \bf{x}_-\\[1mm]
2\lambda \bf{x}_+ 
\end{bmatrix} = O.\label{tilde}
\end{align}
From the injectivity \eqnref{injK} and the notation \eqnref{X2'}, we get the injectivity of $\left(\lambda^2 I - \mathcal{B} \right)$ as follows:
$$
\left(\lambda^2 I - \mathcal{B} \right) \begin{bmatrix}
S^* \bf{x}_-\\[1mm]
2\lambda \bf{x}_+ 
\end{bmatrix} = O
\implies
\begin{bmatrix}
\bf{x}_-\\[1mm]
\bf{x}_+ 
\end{bmatrix}
=O.
$$
Since $S^*$ is invertible and $\lambda$ is a real number,  the injectivity of $\left(\lambda I - \big[\mathbb{K}_\dA^*\big]\right)$ implies those of $\left(\lambda^2 I - \mathcal{B} \right)$:
$$
\left(\lambda^2 I - \mathcal{B} \right) \widetilde{X} = O
\implies \widetilde{X}=O,
$$
where $\widetilde{X} = \begin{bmatrix}
S^* \bf{x}_-\\[1mm]
2\lambda \bf{x}_+ 
\end{bmatrix}$ and $\widetilde{X}$ spans the domain of the operator $\left(\lambda^2 I - \mathcal{B} \right)$.

By the equivalent relations \eqnref{yx} and \eqnref{tilde}, one finds that the injectivity of $\left(\lambda^2 I - \mathcal{B}\right)$ implies those of $\left(\lambda I - \big[\mathbb{K}_\dA^*\big]\right)$.  Therefore, we get the desired result.
\end{proof}

Based on Lemma \ref{injectivity}, we obtain the following theorem.

\begin{theorem}\label{eigenvalueK}
Let $\lambda\in(-1/2,1/2)$. Then $\lambda$ is an eigenvalue of $\mathbb{K}_\dA^*$ if and only if $\lambda^2$ is an eigenvalue of $\mathcal{B}$. 
\end{theorem}
\begin{proof}
The contraposition of Lemma \ref{injectivity} is that, for given $\lambda$,  the operator
$\left(\lambda I - \big[\mathbb{K}_\dA^*\big]\right)$ is not injective if and only if $\left(\lambda^2 I - \mathcal{B}\right)$ is not injective.  Thus, $\lambda$ is an eigenvalue of $\mathbb{K}_\dA^*$ if and only if $\lambda^2$ is an eigenvalue of $\mathcal{B}$.
\end{proof}

We will use the following version of the Gershgorin circle theorem.  The first statement follows from~\cite[p.\,39]{Chonchaiya:2010:CSP}.
%A few studies \cite{Chonchaiya:2010:CSP, Farid:1991:SPD, Salas:1999:GTM, Shivakumar:1987:EIM} discussed the generalization of the Gershgorin circle theorem to infinite matrices.

\begin{lemma}\label{Gershgorin_type}
Let $\mathcal{B} = [b_{mn}]_{m,n\in\ZZ\setminus\{0\}}$ be the infinite matrix defined in \eqnref{altmatrix}, which operates on all spaces $\ell^p(\ZZ\setminus\{0\})$, $1\le p \le \infty$. The eigenvalues of $\mathcal{B}$ lie in the union of the Gershgorin disks, that~is,
$$
\sigma_{\mathrm{point}}(\mathcal{B}) \subseteq \bigcup_{m\in\ZZ\setminus\{0\}} \mathcal{G}_m,
$$
where $\mathcal{G}_m$ is the Gershgorin disk defined by
$$
\mathcal{G}_m = \left\{ \mu\in\CC: |\mu - b_{mm} | \le \max\Big[\sum_{{n\in\ZZ\setminus\{0\}}, n\neq m} |b_{mn}|,  \sum_{{n\in\ZZ\setminus\{0\}}, n\neq m} |b_{nm}| \Big] \right\}.
$$
For each $m\in\ZZ\setminus\{0\}$, if $\mathcal{G}_m$ is disjoint from $\mathcal{G}_n$ for all $n\neq m$, then $\mathcal{G}_m$ contains an element of $\sigma_{\mathrm{point}}(\mathcal{B})$, that is, there exist an eigenvalue $\lambda_m$ of $\mathcal{B}$ such that $\lambda_m \in \mathcal{G}_m$.
\end{lemma}

\begin{proof}
We refer the reader to \cite{Chonchaiya:2010:CSP} for using the Gershgorin circle theorem to infinite matrices. Since non-diagonal terms of $\mathcal{B}$ decay exponentially,  we can apply \cite[p.\,39]{Chonchaiya:2010:CSP} to conclude that all eigenvalues of $\mathcal{B}$ lie within the union of the Gershgorin disks.

The second statement follows from a standard continuity argument.
Let $D$ denote the diagonal part of $\mathcal{B}$, and define $\mathcal{B}_t = t \mathcal{B} + (1-t)D$ for $t\in[0,1]$ so that $\mathcal{B}_0 = D$ and $\mathcal{B}_1 = \mathcal{B}$. The eigenvalues $\{\lambda_{m,0}\}_{m\in\ZZ\setminus\{0\}}$, of $\mathcal{B}_0$ are the diagonal elements.
Let $\mathcal{G}_{m,t}$ denote the Gershgorin disk of $\mathcal{B}_t$, centered at the eigenvalue $\lambda_{m,t}$.  Since the sets $\{\mathcal{G}_{m,1}\}$ are disjoint and each disk $\mathcal{G}_{m,t}$ shrinks monotonically to the single point $\{\lambda_{m,0}\}$ as $t\to0$, the sets $\{\mathcal{G}_{m,t}\}$, for each $t\in[0,1]$, are also disjoint.  Continuity of $\lambda_{m,t}$ in $t$ now guarantees that $\lambda_{m,t}$ must remain in the disk $\mathcal{G}_{m,t}$ for all~$t\in[0,1]$.
%For all $t$, the eigenfunctions of $\mathcal{B}_t$ are in $\ell^2(\ZZ\setminus\{0\})$ and the Gershgorin circle theorem asserts that the corresponding eigenvalues of $\mathcal{B}_t$ are contained in the union of the Gershgorin disks. Let $\{\lambda_{m,t}\}_{m\in\ZZ\setminus\{0\}}$ denote the eigenvalues of $\mathcal{B}_t$. For all $t$, the each eigenvalue $\lambda_{m,t}$ is analytic with respect to $t$ and there exists Gershgorin disks $\mathcal{G}_{m,t}$ of $\mathcal{B}_t$ such that
%$$
%\lambda_{m,t} \in \bigcup_{m\in\ZZ\setminus\{0\}} \mathcal{G}_{m,t}.
%$$
%For $t=0$, we note that $\mathcal{G}_{m,0} = \{\lambda_{m,0}\}$, $\mathcal{G}_{m,1} = \mathcal{G}_m$, and $ \mathcal{G}_{m,1}$ are mutually disjoint for all $m$.  Since $\mathcal{G}_{m,t}$ converges to a point set $\mathcal{G}_{m,0}$ as $t\to0$, the Gershgorin disks $\{\mathcal{G}_{m,t}\}_{m\in\ZZ\setminus\{0\}}$ are also mutually disjoint for any $t$. Therefore, each eigenvalue $\lambda_{m,t}$ must stay in $\mathcal{G}_{m,t}$ for all $t\in[0,1]$.  If we take $t=1$, then we get the result as desired.
\end{proof}

\begin{lemma}%[\cite{Kuhnau:2006:NMC}]
\label{grunsky_bound}
Let $G = (g_{mn})_{m,n=1}^\infty$ with $g_{mn} = c_{mn}/r_\iii^{m+n}$.
%Suppose that $\Gamma_\iii$ is an analytic boundary.
There exists $\rho\in[0,1)$ such that $|g_{mn}| \le \sqrt{m} \, \rho^{m+n}$ for all $m,n\in\NN$.
\end{lemma}
\begin{proof}
$\Gamma_\iii$ is the image of a circle under the conformal map $z=\Psi(w)$ and thus has an analytic and univalent continuation to $|w|>r_\iii-\delta$ for a small enough $\delta>0$.  The inequality \eqnref{Gineq} yields
$$
 \sqrt{\frac{n}{m}} \frac{|c_{mn}|}{(r_\iii-\delta)^{m+n}} < 1,
$$
and thus
$$|g_{mn}| = \frac{|c_{mn}|}{r_\iii^{m+n}} < \sqrt{\frac{m}{n}} \frac{(r_\iii-\delta)^{m+n}}{r_\iii^{m+n}} < \sqrt{m}\, \rho^{m+n},$$
where $\rho = (r_\iii-\delta)/r_\iii$.
\end{proof}

%\begin{lemma}\label{asymptotics}
%Let $\rho,r\in[0,1)$. We have
%$$
%\sum_{m=1}^\infty \sum_{n=1}^\infty \rho^{2m+2n} (1-r^{2n})^2 = \left(\frac{\rho^2}{1-\rho^2}\right) \left(\frac{\rho^2}{1-\rho^2}- \frac{2\rho^2 r^2}{1-\rho^2 r^2} + \frac{\rho^2 r^4}{1-\rho^2 r^4}\right) =: f_\rho(r),
%$$
%and $f_\rho$ has the asymptotics,
%$$
%f_\rho(r) = o(1-r) \quad \mbox{as } r\to1^-,
%$$
%where $o(\cdot)$ denotes the little--$o$ notation.
%\end{lemma}
%\begin{proof}
%From the geometric series expansion, one can easily get the series value.  The little--$o$ asymptotics holds from the L'h\^{o}pital's rule. 
%\end{proof}

\begin{theorem}\label{estimate}
Let $r=\frac{r_\iii}{r_\eee}<1$ and $\rho\in(0,1)$ be given, and let $\lambda$ be an eigenvalue of $\mathbb{K}_\dA^*$ with $\lambda\neq \pm\frac{1}{2}$.  Then
\begin{align*}
\lambda^2 \in \bigcup_{m=1}^\infty B(m,r),
\end{align*}
where $B(m,r)$ is a disk defined by
\begin{align*}
B(m,r) = \left\{\mu\in\left[0,1/4\right]: \left| \mu - \frac{r^{2m}}{4} \right| \le R(m,r) \right\}
\end{align*}
with a radius given by
\beq\label{rad}
R(m,r) = \sqrt{m} \, \rho^m (1-r^2) M(m,r)
\eeq
and $M :=\max(M_1^\mathrm{row}, M_1^\mathrm{col}, M_2^\mathrm{row}, M_2^\mathrm{col})$ is a bounded function of $m$ and $r$ with
\begin{align*}
&M_1^\mathrm{row}(m,r) = \frac{\rho}{4(1-\rho)(1-\rho r^2)},\\ 
&M_1^\mathrm{col}(m,r) = \frac{1}{4\sqrt{m}} \left(\frac{1-r^{2m}}{1- r^2}\right)r^{2m} \sum_{n=1}^\infty \sqrt{n} \rho^n,\\
&M_2^\mathrm{row}(m,r) = \frac{1}{4} \sum_{n=1}^\infty \rho^{n} \left(\frac{1-r^{2n}}{1- r^2}\right)\Big( r^{2n} + \sum_{k=1}^\infty \sqrt{n} \rho^{n+k} (1-r^{2k}) \Big),\\ 
&M_2^\mathrm{col}(m,r) = \frac{1}{4}\left(\frac{1-r^{2m}}{1- r^2}\right) \sum_{n=1}^\infty \sqrt{n} \, \rho^n \Big(\frac{1}{\sqrt{m}} + \sum_{k=1}^\infty  \rho^{m+k} (1-r^{2k}) \Big).
\end{align*}
\end{theorem}
\begin{proof}
Applying Lemma \ref{Gershgorin_type} on the first and the second block rows of \eqnref{altmatrix}, 
$$\bigcup_{m\in\NN,  \, i=1,2} B_i(m,r)$$
contains the spectrum of the block matrix \eqnref{altmatrix}, where $B_1(m,r)$ and $B_2(m,r)$ are denoted by
\beq\label{B1}
B_1(m,r) = \left\{ \mu\in\CC: \left| \mu - \frac{1}{4}\left[ r^{2\NN} \right]_{mm} \right| \le \max\left[R_{1}^\mathrm{row}(m,r),  R_{1}^\mathrm{col}(m,r)\right] \right\},
\eeq
with
\beq\label{B1R}
R_{1}^\mathrm{row}(m,r) = \sum_{n=1}^\infty \left| \frac{1}{4} \left[ - G(I-r^{2\NN}) \right]_{mn} \right|, \quad R_{1,m}^c = \sum_{n=1}^\infty \left| \frac{1}{4} \left[ - \overline{G}(I-r^{2\NN})r^{2\NN} \right]_{nm} \right|,
\eeq
and
\begin{align}
B_2(m,r)=
\bigg\{\mu\in\CC:
&\left| \mu - \frac{1}{4} \left[ r^{2\NN} + \overline{G} (I - r^{2\NN}) G (I - r^{2\NN}) \right]_{mm} \right|  \le \max(R_{2,m}^r,R_{2,m}^c)\bigg\},\label{B2}
\end{align}
with
\begin{align}
&R_{2}^\mathrm{row}(m,r) = \sum_{n=1}^\infty \left|  \frac{1}{4} \left[ -\overline{G} (I - r^{2\NN}) r^{2\NN} \right]_{mn} \right| +\sum_{n\neq m} \left| \frac{1}{4} \left[ \overline{G} (I - r^{2\NN}) G (I - r^{2\NN}) \right]_{mn} \right|, \label{B2Rr}\\
&R_{2}^\mathrm{col}(m,r) = \sum_{n=1}^\infty \left|  \frac{1}{4} \left[ -G (I - r^{2\NN}) \right]_{nm} \right| +\sum_{n\neq m} \left| \frac{1}{4} \left[ \overline{G} (I - r^{2\NN}) G (I - r^{2\NN}) \right]_{nm} \right|.\label{B2Rc}
\end{align}

From Lemma \ref{grunsky_bound}, the radii satisfy
\begin{align}
&R_{1}^\mathrm{row}(m,r) =  \frac{1}{4} \sum_{n=1}^\infty |g_{mn}| (1-r^{2n}) \le \frac{1}{4} \sum_{n=1}^\infty \sqrt{m} \, \rho^{m+n} (1-r^{2n}) = \sqrt{m} \, \rho^m (1-r^2) M_1^\mathrm{row}(m,r)\label{R1r}\\
&R_{1}^\mathrm{col}(m,r) = \frac{1}{4} \sum_{n=1}^\infty |g_{nm}| (1-r^{2m})r^{2m} \le \frac{1}{4} \sum_{n=1}^\infty \sqrt{n} \, \rho^{m+n} (1-r^{2m})r^{2m} = \sqrt{m} \,  \rho^{m} (1-r^2) M_1^\mathrm{col}(m,r),\label{R1c}
\end{align}
where $M_1^\mathrm{row}(m,r) = \frac{\rho}{4(1-\rho)(1-\rho r^2)}$ and $M_1^\mathrm{col}(m,r) = \frac{r^{2m}(1-r^{2m})}{4\sqrt{m} (1-r^2)}  \sum_{n=1}^\infty \sqrt{n} \rho^n$. If $\mu\in B_1(m,r)$, then  \eqnref{R1r}, \eqnref{R1c}, and \eqnref{B1} implies that
\beq\label{disk1}
\left| \mu - \frac{r^{2m}}{4} \right| \le \sqrt{m} \, \rho^m (1-r^2) \max\left[M_1^\mathrm{row}(m,r), M_1^\mathrm{col}(m,r)\right].
\eeq

The inequality in \eqnref{B2} and the triangle inequality yields 
\beq\label{B2'}
\left| \mu - \frac{r^{2m}}{4} \right| = \left| \mu - \frac{1}{4} \left[ r^{2\NN} \right]_{mm} \right|\le \max(R_2^\mathrm{row},R_2^\mathrm{col}) + \left|\frac{1}{4} \left[\overline{G} (I - r^{2\NN}) G (I - r^{2\NN}) \right]_{mm} \right|.
\eeq
Let's estimate the right-hand side of \eqnref{B2'}.
\begin{align}
& R_{2}^\mathrm{row}(m,r) + \left|\frac{1}{4} \left[\overline{G} (I - r^{2\NN}) G (I - r^{2\NN}) \right]_{mm} \right| \notag\\
&= \sum_{n=1}^\infty \left|  \frac{1}{4} \left[ -\overline{G} (I - r^{2\NN}) r^{2\NN} \right]_{mn} \right|
+ \sum_{n=1}^\infty \left| \frac{1}{4} \left[ \overline{G} (I - r^{2\NN}) G (I - r^{2\NN}) \right]_{mn} \right|\notag\\
&\le \frac{1}{4} \sum_{n=1}^\infty |g_{mn}| (1-r^{2n})\Big( r^{2n} + \sum_{k=1}^\infty |g_{nk}|(1-r^{2k}) \Big)\notag\\
&\le \frac{1}{4} \sum_{n=1}^\infty \sqrt{m} \, \rho^{m+n} (1-r^{2n})\Big( r^{2n} + \sum_{k=1}^\infty \sqrt{n} \rho^{n+k} (1-r^{2k}) \Big)\notag\\
&= \sqrt{m} \, \rho^m (1-r^2) M_2^\mathrm{row}(m,r) \label{R2r}
\end{align}
where
$$M_2^\mathrm{row}(m,r) = \frac{1}{4} \sum_{n=1}^\infty \rho^{n} (\frac{1-r^{2n}}{1-r^2})\Big( r^{2n} + \sum_{k=1}^\infty \sqrt{n} \rho^{n+k} (1-r^{2k}) \Big),$$ 
and
\begin{align}
& R_{2}^\mathrm{col}(m,r) + \left|\frac{1}{4} \left[\overline{G} (I - r^{2\NN}) G (I - r^{2\NN}) \right]_{mm} \right| \notag\\
&= \sum_{n=1}^\infty \left|  \frac{1}{4} \left[ -G (I - r^{2\NN}) \right]_{nm} \right|
+ \sum_{n=1}^\infty \left| \frac{1}{4} \left[ \overline{G} (I - r^{2\NN}) G (I - r^{2\NN}) \right]_{nm} \right|\notag\\
&\le \frac{1}{4} \sum_{n=1}^\infty |g_{nm}| (1-r^{2m})\Big( 1 + \sum_{k=1}^\infty |g_{mk}|(1-r^{2k}) \Big)\notag\\
&\le \frac{1}{4} \sum_{n=1}^\infty \sqrt{n} \, \rho^{m+n} (1-r^{2m})\Big(1 + \sum_{k=1}^\infty \sqrt{m} \rho^{m+k} (1-r^{2k}) \Big)\notag\\
&= \sqrt{m} \, \rho^m (1-r^2) M_2^\mathrm{col}(m,r),\label{R2c}
\end{align}
where $$M_2^\mathrm{col}(m,r) = \frac{1}{4}(\frac{1-r^{2m}}{1-r^2}) \sum_{n=1}^\infty \sqrt{n} \, \rho^n \Big(\frac{1}{\sqrt{m}} + \sum_{k=1}^\infty \rho^{m+k} (1-r^{2k}) \Big).$$
If $\mu\in B_2(m,r)$, then  \eqnref{R2r}, \eqnref{R2c}, and \eqnref{B2'} implies that
\beq\label{disk2}
\left| \mu - \frac{r^{2m}}{4} \right| \le \sqrt{m} \, \rho^m (1-r^2) \max\left[M_2^\mathrm{row}(m,r), M_2^\mathrm{col}(m,r)\right].
\eeq

By combining \eqnref{disk1} and \eqnref{disk2}, we get
\beq\label{UB}
\bigcup_{m\in\NN,  \, i=1,2} B_i(m,r) \subset \bigcup_{m=1}^\infty B(m,r).
\eeq
Then Lemma \ref{Gershgorin_type},  Theorem \ref{eigenvalueK}, and \eqnref{UB} give the desired result.
\end{proof}

The ball $B(m,r)$ in Theorem \ref{estimate} is disjoint from all other balls for sufficiently large~$m$.  The disjointness guarantees existence of an eigenvalues in the ball.

\begin{lemma}\label{disjoint}
Let $r$ and $\rho$ be given such that $(\frac{1+\rho}{2})^{\frac{1}{2}}< r < 1$.  There exists $m_0 = m_0(\rho) \in\NN$ such that the following properties hold for all $m\ge m_0$:
\begin{enumerate}[label=$(\alph*)$]
\item $B(m,r)$ is disjoint from $B(n,r)$ for all $n\in\NN,  \, n\neq m$. 
\item $B(m,r)$ contains at least one squared eigenvalue of $\mathbb{K}_\dA^*$.
\end{enumerate}
\end{lemma}

\begin{proof}
Using Lemma \ref{Gershgorin_type} and Theorem \ref{estimate},  we see that $(a)$ implies $(b)$.  For the proof of $(a)$,  it is enough to show that $B(m,r)$ is disjoint with the consecutive balls, that is, for some $m_0\in\NN$,
\beq\label{disjointness}
B(m,r) \cap B(m\pm 1,r)= \emptyset, \quad m\ge m_0.
\eeq
To do this, we will show that the distance between the centers of $B(m,r)$ and $B(m\pm1,r)$ is greater than the sum of the radii of the consecutive balls. The distance between the centers~is
$$
d := \left| \frac{r^{2m}}{4} - \frac{r^{2(m\pm1)}}{4} \right| = \frac{r^{2m} \left| 1 - r^{\pm2}\right| }{4}.
$$
As $r > (\frac{1+\rho}{2})^{\frac{1}{2}}$, we have
\beq\label{ineq22}
4d = r^{2m} \left| 1 - r^{\pm2}\right|  > \bigg( \frac{1+\rho}{2} \bigg)^m \left| 1 - r^{\pm2}\right|.
\eeq
From \eqnref{rad}, the sum of the radii of $B(m,r)$ and $B(m\pm1)$ satisfies
\beq\label{rrho}
R(m,r) + R(m\pm1,r) 
= (\sqrt{m} \, M(m,r) + \sqrt{m\pm1} \, \rho^{\pm1} M(m\pm 1,r)) (1-r^2) \rho^m \le C_{\rho} \sqrt{m+1} \, \rho^m (1-r^2)
\eeq
for some constant $C_\rho > 0$ independent of $r$ and $m$.  Then \eqnref{ineq22} and \eqnref{rrho} gives
\beq \label{bound}
R(m,r) + R(m\pm1,r) \le C_{\rho} \sqrt{m+1} \, \rho^m (1-r^2) <  4d \, C_\rho \sqrt{m+1} \left( \frac{2\rho}{1+\rho} \right)^m
\eeq
by the inequality
$$
\frac{1-r^2}{\left| 1 - r^{\pm2}\right|}
\le \max\left(
\frac{1-r^2}{1-r^2},
\frac{1-r^2}{r^{-2}-1}
\right)
= \max\left(1,r^2\right)
\le 1.
$$
Since $\frac{2\rho}{1+\rho} < 1$, the right--hand side of \eqnref{bound} eventually converges to zero as $m$ grows. Hence, there exist $m_0\in\NN$ such that
$$
R(m,r) + R(m\pm1,r) \le d \quad \mbox{for all }m\ge m_0.
$$
Therefore, \eqnref{disjointness} holds and it completes the proof.
\end{proof}

\begin{theorem}\label{conv_eig}
The Hausdorff distance between the spectrum of the NP operator of $A$ and $[-1/2,1/2]$ converges to zero as $r_\eee$ approaches to $r_\iii$, that is, 
$$
\lim_{r \to 1} d_H\left(\sigma(\mathbb{K}_\dA^*),[-1/2,1/2]\right) = 0.
$$
\end{theorem}
\begin{proof}
Let us denote $\sigma(\mathbb{K}_\dA^*)^2 := \{ \lambda^2: \lambda \in \sigma(\mathbb{K}_\dA^*) \}$.
It is enough to show that
\beq\label{square}
\lim_{r \to 1} d_H\left(\sigma(\mathbb{K}_\dA^*)^2,[0,1/4]\right) = 0
\eeq
because spectra of the NP operator of two-dimensional domains have the twin relation \cite{Mayergoyz:2005:ERN}, that is,  $-y\in\sigma(\mathbb{K}_\dA^*)$ if $y\in\sigma(\mathbb{K}_\dA^*)$. Thus,  \eqnref{square} is equivalent with
$$
\lim_{r\to1} d_H\left(\sigma(\mathbb{K}_\dA^*),[-1/2,1/2]\right) = 0.
$$

Let us prove \eqnref{square}. Let $r \in (r_0, 1)$ with $r_0 = (\frac{1+\rho}{2})^{\frac{1}{2}}$.  For any $m\ge m_0$, there exists $\lambda_m \in\sigma\left(\mathbb{K}_\dA^*\right)$ satisfying $\lambda_m\in B(m,r)$ by using Theorem \ref{estimate} and Lemma \ref{disjoint}, i.e.,
$$
\left|\lambda_m^2 - \frac{r^{2m}}{4} \right| < R(m,r), \quad m\ge m_0.
$$
From \eqnref{rad}, we have
\begin{align}
\left|\lambda_m^2 - \frac{r^{2m}}{4}\right| < \sqrt{m} \, \rho^m (1-r^2) M(m,r).
\label{triangle}
\end{align}
For a given $\epsilon>0$, we can find $m_1\ge m_0$ such that, for all $m\ge m_1$,
\beq\label{ineq1}
\sqrt{m} \, \rho^m (1-r^2) M(m,r) < \frac{\epsilon}{2}.
\eeq
We may choose $r_1<1$ such that, if $0 \le m<m_1$,
\beq\label{ineq4}
\frac{r^{2m}}{4} - \frac{r^{2m_1}}{4} \le \frac{1}{4} - \frac{r^{2m_1}}{4}  < \frac{\epsilon}{4} \quad \mbox{for any } r>r_1.
\eeq
Moreover, there is $r_2 = (1-\epsilon)^{\frac{1}{2}}\in(0,1)$ such that
\beq\label{ineq2}
\frac{r^{2m}}{4} - \frac{r^{2m+2}}{4} = \frac{r^{2m}(1-r^2)}{4} \le \frac{1-r^2}{4} < \frac{\epsilon}{4}.
\eeq
for all $r\in(r_2,1)$ and $m\ge0$.  From now on, we denote $r_{\max} = \max(r_0,r_1,r_2)$. 

Let $\mu$ be an arbitrary number in $[0,1/4]$.  From \eqnref{ineq2} and the pigeonhole principle,  we can find some integer $m'\ge0$ such that
\beq\label{ineq3}
\left| \mu - \frac{r^{2m'}}{4} \right| < \frac{\epsilon}{4}.
\eeq
If $m'\ge m_1$, then \eqnref{triangle}, \eqnref{ineq1}, and \eqnref{ineq3} imply that, for $r > r_{\max}$,
\beq\label{epsil1}
\left| \mu - \lambda_{m'}^2 \right| \le \left| \mu - \frac{r^{2m'}}{4} \right| + \left|  \lambda_{m'}^2 - \frac{r^{2m'}}{4} \right| < \frac{\epsilon}{4} + \frac{\epsilon}{2} < \epsilon.
\eeq
Otherwise,  $m'< m_1$,  \eqnref{ineq3}, \eqnref{ineq4}, \eqnref{triangle}, and \eqnref{ineq1} yields that, for $r>r_{\max}$,
\beq\label{epsil2}
\left| \mu - \lambda_{m_1}^2 \right| \le \left| \mu - \frac{r^{2m'}}{4} \right| + \left|  \frac{r^{2m'}}{4} - \frac{r^{2m_1}}{4} \right| + \left|  \lambda_{m_1}^2 - \frac{r^{2m_1}}{4} \right|  < \frac{\epsilon}{4} + \frac{\epsilon}{4} + \frac{\epsilon}{2} = \epsilon.
\eeq

Let $D_\epsilon(x)$ denote a disk of radius $\epsilon$ centered at $x$.  From \eqnref{epsil1} and \eqnref{epsil2}, we always have $\lambda\in\sigma(\mathbb{K}_\dA^*)$ such that $\mu\in D_\epsilon(\lambda^2)$ for any given $\mu\in[0,1/4]$. As $[0,1/4]$ already contains $\sigma(\mathbb{K}_\dA^*)^2$, the Hausdorff distance between $\sigma(\mathbb{K}_\dA^*)^2$ and $[0,1/4]$ turns out that, if $r>\max(r_0,r_1,r_2)$,  
$$
d_H\left(\sigma(\mathbb{K}_\dA^*)^2,[0,1/4]\right)=\inf\left\{ \delta\ge0: [0,1/4] \le \bigcup_{\lambda\in\sigma(\mathbb{K}_\dA^*)} D_\delta(\lambda^2) \right\} < \epsilon
$$
for an arbitrary $\epsilon>0$. Therefore, we conclude that
$$
\lim_{r\to1} d_H\left(\sigma(\mathbb{K}_\dA^*)^2,[0,1/4]\right) = 0.
$$
The Theorem follows by the symmetry of $\sigma(\mathbb{K}_\dA^*)\setminus\{1/2\}$.
\end{proof}

Finally, (\ref{triangle}) and (\ref{ineq1}) show that as $r\to1$, $\sigma(\Kcal^*_\Gamma)$ tends to the spectrum of the NP operator for a circular annulus.  Indeed, the eigenvalues for circular annulus with ratio of outer an inner radii equal to~$r$ are $\{\pm r^m/2\}$ and $1/2$.

%{\color{blue}
%Indeed, we have The eigenvalues of $\mathbb{K}_{\Ome\setminus\overline{\Om}_i}^*$ have the form
%\begin{align*}
%&\left|\lambda^2 - \frac{r^{2m}}{4}\right| < \left| \frac{r^{2(m+1)}}{4} - \frac{r^{2m}}{4} \right|  =  O\left(r^{2m} (1-r) \right),  \quad \mbox{as } m\to \infty, r\to 1.
%\end{align*}
%}

\section{Numerical results}

We compute eigenvalues of a finitely approximated matrix of $\mathbb{K}_{\Ome\setminus\overline{\Om}_i}^*$ for the thin distorted annulus $A$ introduced in Figure \ref{doublyctd}.  The exterior conformal mapping~is
$$\Psi(w) = w + \frac{0.3+0.5i}{w} - \frac{0.2+0.1i}{w^3} + \frac{0.1i}{w^5} + \frac{0.05i}{w^6} + \frac{0.01i}{w^7},$$
and the interior and exterior radii are $r_\iii = 1.1$ and $r_\eee = 1.15$.
According to proof of Theorem~\ref{conv_eig}, the spectrum of the Neumann-Poincar\'{e} operator $A$ approaches  \beq\label{circ_eig}
\left\{ \pm\frac{r^m}{2} \right\}_{m\in\NN}
\eeq as $r = r_\iii/r_\eee \to 1$. Indeed,  \eqnref{circ_eig} is the spectrum of the circular annulus with the interior and exterior radii given by $r_\iii$ and $r_\eee$. The following Figure \ref{eig_comparing} compares eigenvalue distributions of both annuli.
\begin{figure}[H]
\centering
\includegraphics[scale=0.6, trim = 2cm 9cm 2cm 8cm, clip]{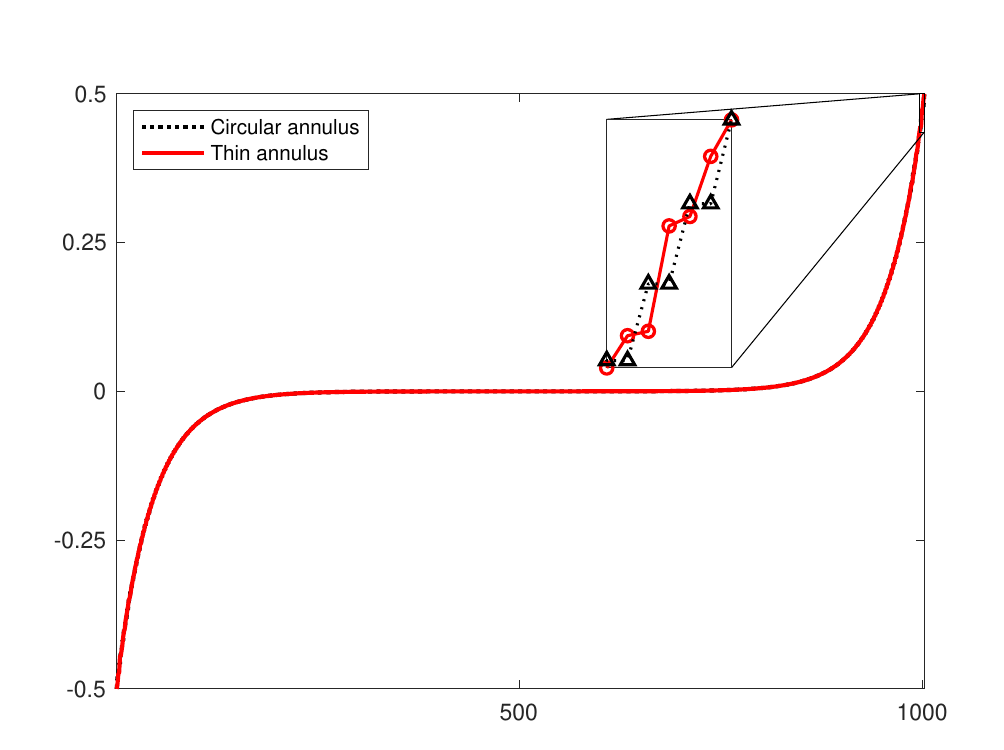}
\caption{Approximated eigenvalue distributions for the circular and the distorted annuli. We truncate $\mathbb{K}_{\Ome\setminus\overline{\Om}_i}^*$ as a finite matrix of size $1002\times1002$. The $x$-- and $y$--axes represent an index and the corresponding eigenvalue, respectively. There is a slight difference in the eigenvalues around $\pm1/2$, but as the eigenvalues approach 0, the distributions are fairly consistent.}
\label{eig_comparing}
\end{figure}
\begin{figure}[H]
\centering
\includegraphics[scale=0.6, trim = 2cm 9cm 2cm 8cm, clip]{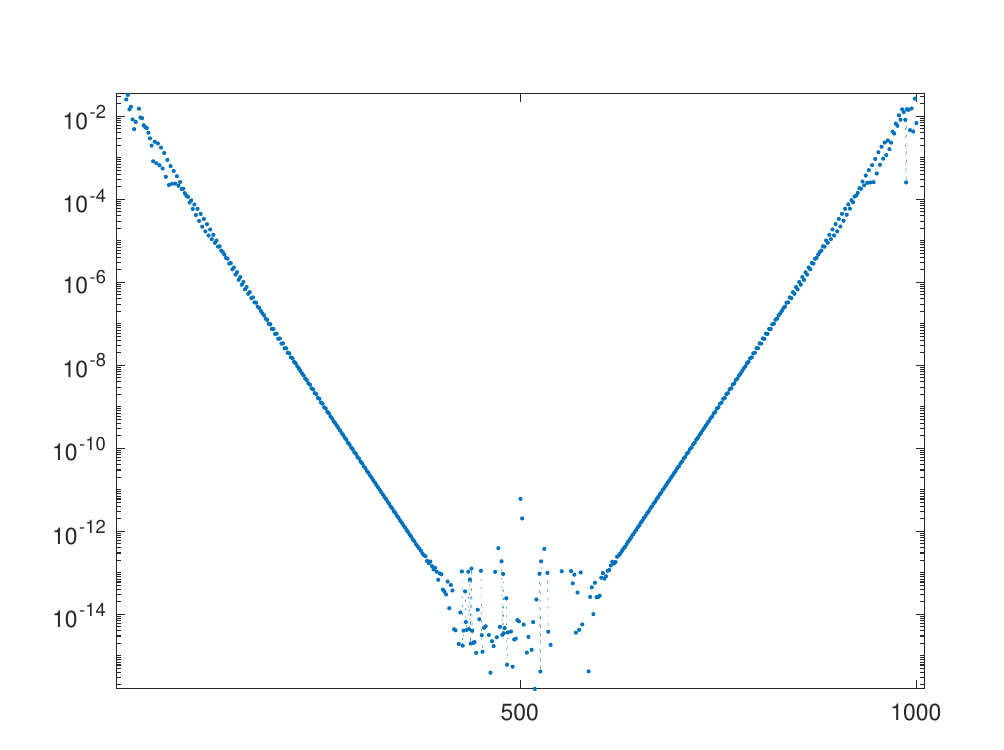}
\caption{A distribution for the difference of eigenvalues.}
\label{Difference}
\end{figure}
Figure \ref{Difference} shows the error of the difference of the eigenvalues in Figure \ref{eig_comparing}, that is, 
$$
\frac{\hat{\lambda}_n^\pm - \lambda_n^{\pm,\text{disk}} }{\lambda_n^{\pm,\text{disk}}}
$$
where $1\le n \le 1002$ and $$\lambda_n^{\pm, \text{disk}} \in \left\{ \pm\frac{1}{2}\left( \frac{r_\iii}{r_\eee} \right)^m \right\}_{0\le m \le500}.
$$
The difference is apparently exponentially decreasing near zero.

\bibliographystyle{plain}

\end{document}